\documentclass[12pt,leqno,twoside,a4paper]{article}
\usepackage{amsfonts,amsmath,amsthm,amssymb}
\usepackage[english]{babel}
\usepackage{color}
\newtheorem{theorem}{Theorem}[section]
\newtheorem{prop}[theorem]{Proposition}
\newtheorem{lemma}[theorem]{Lemma}
\newtheorem{corollary}[theorem]{Corollary}
\theoremstyle{definition}
\newtheorem{definition}[theorem]{Definition}
\newtheorem{definitions}[theorem]{Definitions}

\newtheorem{examples}[theorem]{Examples}
\theoremstyle{remark}
\newtheorem{remark}[theorem]{\bf Remark}
\newtheorem{remarks}[theorem]{\bf Remarks}

 \numberwithin{equation}{section}

\newcommand{\si}{\sigma}

\newcommand{\de}{\delta}

\begin{document}
\title{ \bf\normalsize NONCOMMUTATIVE POLYNOMIAL MAPS}
 \author{ {\bf\normalsize  Andr\'{e} Leroy}\\
 \normalsize Universit\'{e} d'Artois,  Facult\'{e} Jean Perrin\\
\normalsize Rue Jean Souvraz  62 307 Lens, France\\
   \normalsize  e.mail: andre.leroy@univ-artois.fr\\}
\date{ }
\maketitle\markboth{ \bf A.Leroy}{ \bf Noncommutative polynomial
maps}

\begin{abstract}
Polynomial maps attached to polynomials of an Ore extension are
naturally defined.  In this setting we show the importance of
pseudo-linear transformations and give some applications.  In particular, 
factorizations of polynomials in an Ore extension over a finite field
$\mathbb F_q[t;\theta]$, where $\theta$ is the Frobenius automorphism, 
are translated into factorizations in the usual polynomial ring $\mathbb F_q[x]$. 
\end{abstract}

\section*{\normalsize INTRODUCTION}

Polynomial maps associated to an element of
an Ore extension $K[t;\si,\de]$ over a division ring $K$ have been considered and 
studied in different papers such as \cite{Co}, \cite{LL1}, \cite{LL2}
 \cite{LL4}, \cite{LLO}, \cite{Or}.  They have been used
in different areas such as: the construction of factorizations of
Wedderburn polynomials (\cite{LO}, \cite{DL}, \cite{HR}), the criterion for 
diagonalisation of matrices over division rings (\cite{LLO}), the solution
of linear differential equations and applications to time-varying systems (e.g. \cite{MB1}, \cite{MB2}), the constructions of new codes
(\cite{BU}, \cite{BGU}).

Pseudo-linear transformations are intimately related to modules
over an Ore extension $A[t;\si,\de]$ (see Section $1$). In the present
paper, we intend to show that they are also useful
tools for studying polynomial maps.

 In Section 1 we introduce the main definitions and give
some examples.  The results presented here concern general Ore 
extensions over rings which are not necessarily division rings.
The pseudo-linear transformations play a crucial role in this section.
They enable us to explain some features of roots of polynomials through
a bimodule structure of some cyclic modules (Cf. comments before Proposition
\ref{ring of endo} and Corollary \ref{the case of T_a} ) 
and they lead to a formula describing the polynomial map attached to a product of polynomials (Cf. 
Theorem \ref{ring homomorphism and generalized product formula}(2)).

Section 2 is devoted to applications.  We first recall
the way of computing "the number of right roots" of an Ore
polynomial with coefficients in a division ring and apply 
this to polynomials of the Ore extension
$\mathbb F_q[t;\theta]$ built over a finite field $\mathbb F_q$ 
with the Frobenius automorphism $\theta$.  
We also attach to every element $f(t) \in 
\mathbb F_q[t;\theta]$ a polynomial denoted $f^{[]}(x)\in \mathbb F_q[x]$ and 
use pseudo-linear transformations to study the strong relations between factors 
of $f(t)\in \mathbb F_q[t;\theta]$ and factors of $f^{[]}(x)\in \mathbb F_q[x]$.  
This makes Berlekamp algorithm available for factorizations of polynomials in $\mathbb F_q[t;\theta]$. 
As another application, we give a form of Hilbert 90 theorem as well as a short and easy 
proof of a generalized version of the so-called "Frobenius law" for computing
 $p$-th powers of a sum in characteristic $p>0$.  The other
applications are generalizations of standard results or give other
perspectives on them.

\section{Polynomial and pseudo-linear maps}

Let $A$ be a ring with $1$, $\sigma$ an endomorphism of $A$ and $\delta$ a 
$\sigma$-derivation of $A$ (i.e. $\delta\in End(A,+)$ and  
$\delta(ab)=\sigma(a)\delta(b)+\delta(a)b$, for $a,b\in A$).
The skew polynomial ring $R:=A[t;\sigma,\delta]$ (a.k.a. Ore extension) 
consists of elements of the form $\sum_{i=0}^{n}a_it^i$ where addition is performed as 
in the classical case and multiplication 
is based on the rule $ta=\sigma(a)t+\delta(a)$, for $a\in A$. 

\begin{definitions}
Let $A$ be a ring, $\sigma$ an endomorphism of $A$ and $\delta$ a
$\sigma$-derivation of $A$.  Let also $V$ stand for a left
$A$-module.
\begin{enumerate}
\item[a)] An additive map $T:V\longrightarrow V$ such that,
for $\alpha \in A$ and $v\in V$,
$$
T(\alpha v)=\sigma(\alpha)T(v) + \delta(\alpha)v.
$$
is called a ($\si,\de$) pseudo-linear transformation (or a ($\si,\de$)-PLT, for short).
\item[b)] For $f(t)\in R=A[t;\si,\de]$ and $a\in A$, we define $f(a)$ to be the
only element in $A$ such that $f(t)-f(a)\in R(t-a)$.
\end{enumerate}
\end{definitions}

In case $V$ is a finite dimensional vector space and $\si$ is an 
automorphism, the pseudo-linear transformations were introduced in \cite{Ja2}.  
They appear naturally in the context
of modules over an Ore extension $A[t;\si,\de]$.  This is
explained in the next proposition.

\begin{prop}
\label{PLT and R modules} Let $A$ be a ring $\si\in End(A)$ and
$\de$ a $\si$-derivation of $A$.   For an additive group $(V,+)$ the
following conditions are equivalent:
\begin{enumerate}
\item[(i)] $V$ is a  left $R=A[t;\si, \de]$-module;
\item[(ii)] $V$ is a left $A$-module and there exists a $(\si,\de)$ pseudo-linear 
transformation $T:V\longrightarrow V$;
\item[(iii)] There exists a ring homomorphism $\Lambda: R\longrightarrow
End(V,+)$.
\end{enumerate}
\end{prop}
\begin{proof}
The proofs are straightforward,  let us nevertheless mention that, for the implication 
$(i)\Rightarrow (ii)$, the ($\si,\de$)-PLT on $V$ is given by the left multiplication by $t$.
%
%
\end{proof}

Let $T$ be a ($\si,\de$)-PLT defined on an $A$-module $V$.  
Using the above notations, 
we define, for $f(t)=\sum_{i=0}^na_it^i\in R$, 
$f(T):=\Lambda (f(t))=\sum_{i=0}^na_iT^i\in End(V,+)$.   We can now state the following 
corollary.  It will be intensively used in the paper.

\begin{corollary}
\label{ring homomorphism defined by a PLT}
For any $f,g\in R=A[t;\si,\de]$ and any pseudo-linear transformation $T$ we have:
$(fg)(T)=f(T)g(T)$.
\end{corollary}

\begin{examples}
\label{examples of PLT}
{\rm
\begin{enumerate}
\item[(1)] If $\si=id.$ and $\de=0$, a pseudo-linear map is an endomorphism of
left $A$-modules.  If $\de=0$, a pseudo-linear map is usually called a ($\si$) 
semi-linear transformation.
\item[(2)] Let $V$ be a free left $A$-module with basis $\beta=\{e_1,\dots ,e_n\}$ and let 
$T:V \longrightarrow V$ be a ($\si, \de$)-PLT.  This gives rise to a ($\si,\de$)-PLT on 
the left $A$-module $A^n$ as follows:  first define $C=(c_{ij})\in M_n(A)$ by
$
T(e_i)=\sum_i^nc_{ij}e_j.
$ 
and extend component-wise $\si$ and $\de$ to the left $A$-module $A^n$.  
We then define a ($\si,\de$)-PLT on $A^n$ by   
$T_C(\underline{v})=\sigma(\underline{v})C+\delta(\underline{v})$, for $\underline{v}\in 
A^n$. In particular, for $n=1$ and $a\in A$, the map $T_a:A \longrightarrow A$ given 
by $T_a(x)=\si(x)a+\de(x)$ is a ($\si,\de$)-PLT.  The map $T_a$ will be called 
the $(\si,\de)$-PLT induced by $a\in A$.  Notice that $T_0=\de$ and $T_1=\si + \de$.
\item[(3)] It is well-known and easy to check that, extending $\si$ and $\de$
from a ring $A$ to $M_n(A)$ component-wise, gives an endomorphism, still denoted $\si$, 
and a $\si$-derivation also denoted $\de$ on the ring $M_n(A)$.  For $n,l\in \mathbb N$ 
we may also extend component-wise $\si$ and $\de$ to the additive group $V:=M_{n\times l}(A)$.  Let 
us denote these maps by $S$ and $D$ respectively.  Then $S$ is a $\si$ semi-linear map 
and $D$ is a ($\si,\de$)-PLT of the left $M_n(A)$-module $V$.  This 
generalizes the fact, mentioned in example $(2)$ above, that $\de$ itself is a 
pseudo-linear transformation on $A$. 
\item[(4)] Let $_AV_B$ be an ($A,B$)-bimodule and suppose that $\si$ and $\de$ are an 
endomorphism and a $\si$-derivation on $A$, respectively. If $S$ is a 
$\si$ semi-linear map and $T$ is a ($\si,\de$) PLT on $_AV$, then for any $b\in B$, the 
map $T_b$ defined by $T_b(v)=S(v)b + T(v)$, for $v\in V$, is a ($\si,\de$) pseudo-linear 
map on $V$.  
\item[(5)] Using both Examples $(3)$ and $(4)$ above, we obtain a $(\si,\de)$ 
pseudo-linear transformation on the set of rectangular matrices $V:=M_{n\times l}(A)$ 
(considered as an ($M_n(A),M_l(A)$)-bimodule)  
by choosing a square matrix $b\in M_l(A)$ and putting 
$T_b(v)=S(v)b+D(v)$ where $S$ and $D$ are defined component-wise as in Example $(3)$ and 
$v\in V$.  This construction will be used in Proposition \ref{ring of endo}.
\end{enumerate}
}\end{examples}

\begin{remarks}
\label{composition and link}
\begin{enumerate}
\item[(1)]
 Let us mention that the composition of pseudo-linear
transformations is usually not a pseudo-linear transformation.
Indeed, let $T:V\longrightarrow V$ be a $(\si,\de)$-PLT.  
For $a\in A$, $v\in V$ and $n\ge 0$, we
have $T^n(av)=\sum_{i=0}^nf^n_i(a)T^i(v)$, where $f^n_i$ is the sum
of all words in $\si$ and $\de$ with $i$ letters $\si$ and
$n-i$ letters $\de$. 
\item[(2)]  Let us now indicate explicitly the link between polynomial 
maps and pseudo-linear transformations.  
Since, for $a\in A$, the pseudo-linear transformation on $A$
associated to the left $R$-module
$V=R/R(t-a)$ is $T_a$ (Cf. Example \ref{examples of PLT}(2)).  The equality 
$f(t).1_V=f(a) + R(t-a)$ leads to 
$$f(T_a)(1)=f(a).$$  
\end{enumerate}
\end{remarks}

\vspace{2mm}

For a left $R$-module $V$, we consider the standard $(R,End_RV)$-bimodule 
structure of $V$.  In the proof of Proposition \ref{PLT and R 
modules} we noticed that $T$ corresponds to the left multiplication by 
$t$ on $V$.  This implies that, for any $f(t) \in R , f(T)$ is a right 
$End_R(V)$-linear map defined on $V$.  In particular, $\ker f(T)$ is a 
right $End_R(V)$ submodule of $V$.  Considering $V=R/R(t-a)$ for $a\in A$, this module 
structure on $\ker (f(T_a))$ explains and generalizes some 
important properties of roots of polynomials obtained earlier (Cf.
\cite{LL1}, \cite{LL2}, \cite{LLO}),  see Corollary \ref{the case of T_a} for 
more details).  
 Let us describe the elements of $End_R(V)$ in 
case $V$ is a free left $A$-module.  We extend the maps $\si$
 and $\delta$ to matrices over $A$ by letting them act on every
 entry.
\begin{prop}
\label{ring of endo}
For $i=1,2$, let $T_i$ be a $(\si,\de)$-PLT
defined on a free $A$-module $V_i$ with basis $\beta_i$ and dimension $n_i$.                                                                                                                                                                                                                                                                                                                                                                                                                                                                                                                                                                                                                                                                                                                                                                                                                                                                                                                                                                                                                                                                                                                                                                                                                                                                                                                                                                                                                                                                                                                                                                                                                                                                                                                                                                                                                                                                                                                                                                                                                                                                                                                                                                                                                                                                       
Suppose $\varphi \in Hom_A(V_1,V_2)$ is an $A$-module homomorphism. 
Let also $B\in M_{n_1\times n_2}(A)$, $C_1\in M_{n_1\times n_1}(A)$ and $C_2 \in 
M_{n_2\times n_2}(A) $ denote 
matrices representing $\varphi$, $T_1$ and $T_2$ 
respectively in the appropriate bases $\beta_1$ and $\beta_2$.
Let $_RV_1$ and $_RV_2$ be the left $R$-module structures induced by $T_1$ and $T_2$, 
respectively.  The following conditions are equivalent: 
\begin{enumerate}
 \item[(i)] $\varphi \in Hom_R(V_1,V_2)$;
 \item[(ii)]  $\varphi T_1=T_2 \varphi$;
 \item[(iii)] $C_1B=\sigma(B)C_2+ \delta(B)$;
 \item[(iv)] $B\in \ker(T_{C_2}-L_{C_1})$ where $T_{C_2}$ (resp. 
 $L_{C_1}$) stands for the pseudo-linear transformation 
(resp. the left multiplication) induced by $C_2$ (resp. $C_1$) on 
 $M_{n_1\times n_2}(A)$ considered as a left $M_{n_1}(A)$-module.
\end{enumerate}
\end{prop}
\begin{proof}
\noindent $(i)\Leftrightarrow (ii)$.  This is clear since, for $i=1,2$, $T_i$ corresponds to the left 
action of $t$ on $V_i$.

\noindent $(ii) \Leftrightarrow (iii)$.  Let us put $\beta_1:=\{e_1,\dots,e_{n_1}\}$,
$\beta_2:=\{f_1,\dots,f_{n_2}\}$, $C_1=(c^{(1)}_{ij})$, $C_2=(c^{(2)}_{ij})$ and $B=(b_{ij})$.  
We then have, for any $1\le i\le n_1$, $T_2(\varphi(e_i))=T_2(\sum_j 
b_{ij}f_j)=\sum_j(\si(b_{ij})T_2(f_j)+\de(b_{ij})f_j)=\sum_k (\sum_j \si(b_{ij})c^{(2)}_{jk} 
+ \delta(b_{ik} ) ) f_k$.  Hence the matrix associated to 
$T_2\varphi$ in the bases $\beta_1$ and $\beta_2$ is $\si(B)C_2 + \delta(B)$.  This yields the result.

\noindent $(iii) \Leftrightarrow (iv)$. 
It is enough to remark that the definition of 
$T_{C_2}$ acting on $M_{n_1\times n_2}(A)$ shows that, for any $B\in M_n(A)$,
$(T_{C_2}-L_{C_1})(B)=\si(B)C_2+\de(B)-C_1B$.  
\end{proof}

\begin{remark}
The above proposition \ref{ring of endo} shows that
the equality $(iii)$ is independent of the bases.  Hence,
if $P_1\in M_{n_1}(A)$ and $P_2\in M_{n_2}(A)$ are invertible matrices associated to 
change of bases in $V_1$ and $V_2$ then $C_1'B'=\sigma(B')C_2'+\delta(B')$ for 
$B':=P_1BP_2^{-1}$, $C_1':=\si(P_1)C_1P_1^{-1} + \de(P_1)P_1^{-1}$ and 
$C_2':=\si(P_2)C_2P_2^{-1}+\de(P_2)P_2^{-1}$.  Of course, this can also be checked 
directly.
\end{remark}



Let $p(t)=\sum_{i=0}^n a_it^i$ be a monic
polynomial of degree $n$ and consider the left $R=A[t;\si,\de]$ module
$V:=R/Rp$.  It is a free left $A$-module with basis $\beta:=\{\overline{1},\overline{t},\dots,
\overline{t^{n-1}}\}$, where $\overline{t^i}=t^i +Rp$ for $i=1,\dots,n-1$.  
In the basis $\beta$, the matrix corresponding to left multiplication by $t$ is the usual companion 
matrix of $p$ denoted by $C(p)$ and defined by
$$ C(p) = \begin{pmatrix}0&1&0&\cdots&0 \\0&0&1&\cdots&0 \\
\vdots&\vdots&\vdots&\ddots&\vdots \\
0&0&0&\cdots&1\\
-a_0&-a_1&-a_2&\cdots&-a_{n-1}
\end{pmatrix}$$

\begin{corollary}
\label{isomorphisme of cyclic modules} 
Let $p_1, p_2\in R=A[t;\si,\de]$ be two monic polynomials of degree $n\ge 1$ with 
companion matrices $C_1,C_2\in M_n(A)$.   $R/Rp_1 \cong R/Rp_2$ if and only if 
there exists an invertible matrix $B$ such that $C_1B=\si(B)C_2 +\de(B)$. 
\end{corollary}
The pseudo-linear transformation induced on $A^n$ by
$C(p)$ will be denoted $T_{p}$.  


Recall that $Rp$ is a two sided ideal in its idealizer ring $Idl(Rp)=\{g\in R \, \vert \,pg\in Rp\}$.
  The quotient ring $\frac{Idl(Rp)}{Rp}$ is called the eigenring of $Rp$ and is isomorphic to 
   $End_R(R/Rp)$.  The $(R, End_R(R/Rp))$-bimodule structure of $R/Rp$
gives rise to a natural $(R, End_R(R/Rp))$-bimodule structure on $A^n$.
For future reference we sum up some information in the form of a corollary.

\begin{corollary}
\label{module structure on kernels}
Let $p(t)\in R$ be a monic polynomial of degree $n$ and denote by $C=C(p)$ 
its companion matrix.  We have:
\begin{enumerate}
\item[(a)] The eigenring $End_R(R/Rp)$ is isomorphic to 
$C_p^{\si,\de}:= \{B\in M_n(A)\, | \, CB=\si(B)C + \de(B\}$.
\item[(b)] $A^n$ has an $(R,C_p^{\si,\de})$-module structure.
\item[(c)] For $f(t)\in R$, $f(T_p)$ is a right $C_p^{\si,\de}$-morphism.
In particular, $\ker f(T_p)$ is a right $C_p^{\si,\de}$-submodule of $A^n$.
\end{enumerate}
\end{corollary} 
We need to fix some notations.  
Thinking of the evaluation $f(a)$ of a polynomial
$f(t)\in R=A[t;\si,\de]$ at $a\in A$ as an element of $A$
representing $f(t)$ in $R/R(t-a)$, we introduce the following
notation:   for a polynomial $f(t)\in R$ and a monic polynomial
$p(t)\in R$ of degree $n$, $f(p)$ stands for the unique element in
$R$ of degree $< \deg (p)=n$ representing $f(t)$ in $R/Rp(t)$.
Since divisions on the right by the monic polynomial $p$ can be
performed in $R$, $f(p)$ is the remainder of the right
division of $f(t)$ by $p(t)$.  We write $\overline{f(p)}$ for the
image of $f(p)$ in $R/Rp$. 
For $v\in V=R/Rp$, we denote by $v_\beta\in A^n$ the row of coordinates of $v$ in the basis 
$\beta:=\{\overline{1},\overline{t},\dots,\overline{t^{n-1}}\}$.
 Using the above notations we can state the following theorem.
\begin{theorem}
\label{ring homomorphism and generalized product formula} Let $p(t)\in R=A[t;\si,\de]$ be a monic 
polynomial of degree $n\ge 1$.  Then:
\begin{enumerate}
\item[(1)] For $f(t)\in R$ we have:
$\overline{f(p)}_\beta=f(T_{p})(1,0,\dots,0).$
\item[(2)] For $f(t),g(t)\in R$, we have:
$\overline{(fg)(p)}_\beta=f(T_{p})(\overline{g(p)}_\beta).$
\item[(3)] For $f(t)\in R$ there exist bijections between 
the following sets $\ker f(T_p)$, $\{g\in R \, \vert \, \deg(g)<n \;{\rm and}\; fg\in 
Rp\}$ and $Hom_R(R/Rf,R/Rp).$
\item[(4)] $Idl (Rp)=\{g\in R \, | \, g(T_p)(1,0,\dots, 0)\in \ker p(T_p)\}$.
\end{enumerate}
\end{theorem}
\begin{proof}
\noindent $(1)$  The definition of $f(p)$ implies that there exists
$q(t)\in A[t;\sigma,\delta]$ such that $f(t)=q(t)p(t) + f(p)$.  This leads
to $f(T_p)=q(T_p)p(T_p)+f(p)(T_p)$.  Since $T_p=T_{c(p)}$ we easily get
$p(T_p)(1,\dots,0)=(0,\dots,0)$.  Noting that $\deg(f(p))<n$, we also have  $f(p)(T_p)(1,0,\dots,0)=\overline{f(p)}_\beta$.  This leads to the required 
equality.

\noindent $(2)$  The point $(1)$ above and corollary \ref{ring homomorphism defined by a 
PLT} give $\overline{(fg)(p)}_\beta=(fg)(T_p)(1,0,\dots,0)=f(T_p)(g(T_p)(1,0,\dots,0))
=f(T_p)(\overline{g(p)}_\beta)$.

\noindent $(3)$  The map $\psi :\ker f(T_p) \longrightarrow R$ defined by 
$\psi ((v_0,\dots,v_{n-1}))=\sum_{i=0}^{n-1}v_it^i$ is injective and, using $(2)$ 
above with $g(t):=\sum_{i=0}^{n-1}v_it^i$, we obtain $0=f(T_p)(v_0,\dots,v_{n-1})
=f(T_p)(\overline{g(p)}_\beta)=\overline{fg(p)}_\beta$.  This means that $fg\in Rp$.
The map $\psi$ is the required first bijection in statement $(3)$.

\noindent Now, if $g\in R$ is such that $\deg(g) < n$ and $fg\in Rp$ then the map 
$\varphi_g:R/Rf\longrightarrow R/Rp$ 
defined by $\varphi_g(h+Rf)=hg+Rp$ is an element of $Hom_R(R/f,R/Rp)$.
The map $\gamma:\{g\in R \,|\,\deg(g)<n, fg\in Rp \}\longrightarrow Hom_R(R/Rf,R/Rp)$ 
defined by $\gamma (g)=\varphi_g$ is easily seen to be bijective.

\noindent $(4)$  Let us remark that $g\in R$ is such that 
$pg\in Rp$ iff $\overline{(pg)(p)}_\beta=0$ iff $p(T_p)(\overline{g(p)}_\beta)=0$.
The first statement $(1)$ above gives the required conclusion.
\end{proof}

The next corollary requires a small lemma which is interesting by itself.
For a free left $A$-module $V$ with basis $\beta=\{e_1,\dots,e_n\}$ and $\varphi \in 
End(V,+)$ we write $\varphi(e_i)=\sum_j\varphi_{ij}e_j$ and denote by $\varphi_\beta\in M_n(A)$ the matrix defined by $\varphi_\beta=(\varphi_{ij})$.
\begin{lemma}
Let $T$ be a pseudo-linear transformation defined on a free left $A$-module $V$ with 
basis $\beta=\{e_1,\dots,e_n\}$ and $f(t)\in R=A[t;\si,\de]$.  Considering $f(t)$ as an 
element of $M_n(A)[t;\si,\de]$, we have $f(T)_\beta=f(T_\beta)$.   
\end{lemma}
\label{matrix of f(T)}
\begin{proof}
 (Cf.  \cite{L} Lemma 3.3).
\end{proof}
The following corollary is an easy generalization of the classical fact that
the companion matrix, $C:=C(p)\in M_n(A)$, of a monic polynomial $p$ of degree 
$n$ annihilates the polynomial itself.  As earlier, we extend $\si$ and $\de$ to
$M_n(A)$ component-wise.
\begin{corollary}
Let $p(t)\in R=A[t;\si,\de]\subset M_n(A)[t;\si,\de]$ be a monic polynomial of degree 
$n> 1$.  Then the following assertions are equivalent:
\begin{enumerate}
\item[(i)] $t\in Idl(Rp)$;
\item[(ii)] for any $f\in R,\;f\in Rp$ if and only if $f(C(p))=0$;
\item[(iii)] $p(C(p))=0$.
\end{enumerate}
\end{corollary}
\begin{proof}
$(i)\Rightarrow (ii)$ Since $t\in Idl(Rp)$, $f\in Rp$ implies $ft^i \in Rpt^i\subset 
Rp$, for any $0\le i \le n-1$.  Theorem \ref{ring homomorphism and generalized product 
formula}$(4)$ then gives $((f(t)t^i)(T_p)(1,0,\dots,0)=(0,\dots,0)$.  Hence, $f(T_p)
(T_p^i(1,0,\dots,0))=(0,\dots,0)$, for $i\in\{0,\dots,n-1\}$.   This leads to 
$f((T_p))_\beta=0$, where $\beta$ is the standard basis of $A^n$.  The above lemma 
\ref{matrix of f(T)} shows that $0=f((T_p))_\beta=f((T_p)_\beta)=f(C)$, where $f(C)$ 
stands for the evaluation of 
$f(t)\in M_n(A)[t;\si,\de]$ at $C$.

\noindent $(ii)\Rightarrow (iii)$  This is clear.

\noindent $(iii)\Rightarrow (i)$ We have $0=p(C(p))=p(T_p)_\beta)=p((T_p)_\beta)=p(T_p)_\beta$.  Since $n>1$, we have, in particular, $(pt)(T_p)(e_1)=p(T_p)((T_p)(e_1))=p(T_p)(e_2)=0$.
  Theorem \ref{ring homomorphism and generalized product formula}(4) implies that $pt\in Rp$, as required.
\end{proof}

Let us sum up all the information that we have gathered in the 
special case where $V=R/R(t-a)$.
When, moreover, $A=K$ is a division ring, these results were proved 
in earlier papers (Cf. \cite{LL1}, \cite{LL2}, \cite{LLO}) using different, 
more computational proofs.    
$U(A)$ stands for the set of invertible elements of $A$.   
For $x\in U(A)$, we denote by $a^x$ the element $\si(x)ax^{-1} + \de(x)x^{-1}$ and 
$\Delta^{\si,\de}(a):=\{a^x\, \vert \, x\in U(A)\}$. 

\begin{corollary}
\label{the case of T_a}
Suppose $a\in A$ and $f,g\in R=A[t;\si,\de]$.  Let $V$ stand for the $R$-module $R/R(t-a)$.  Then: 

\begin{enumerate}
\item[(a)] The map $\Lambda_a: R \longrightarrow End(V,+)$ defined 
by $\Lambda_a(f)=f(T_a)$ is a ring homomorphism.
For $f,g\in R$, we have 
$(fg)(a)=f(T_a)(g(a))$.  
\item[(b)] Suppose $g(a)$ is invertible, then:
$(fg)(a)=f(a^{g(a)})g(a)$.  
In particular, for an invertible element $x\in A$ we have:
$f(T_a)(x)=f(a^x)x$.
\item[(c)] The set $C^{\si,\de}(a):= \{b\in A \, | \, ab=\si(b)a+\de(b)\}$ is a 
ring isomorphic to $End_RV$. 
\item[(d)] If $A$ is a division ring, then so is $C^{\si,\de}(a)$.  In this case, for
any $f(t)\in R$ and any $a\in A$, 
$\ker(f(T_a))=\{x\in A\setminus \{0\} \, |\, f(a^x)=0\}\cup \{0\}$ is a right 
$C^{\si,\de}(a)$-vector space.
\end{enumerate}
\end{corollary}
\begin{proof}
$(a)$  This is a special case of Corollary \ref{ring 
homomorphism defined by a PLT} and Theorem \ref{ring homomorphism and generalized product formula}$(2)$.  

\noindent $(b)$ It is easy to check that, for $x\in U(A)$, $(t-a^x)x=\si(x)(t-a)$.  This 
leads to $f(t)x-f(a^x)x=(f(t)-f(a^x))x\in R(t-a^x)x\subseteq R(t-a)$.  Hence, using $(a)$ 
above with $g(t)=x$, we have $f(a^x)x=(f(t)x)(a)=f(T_a)(x)$.  The other equality is now 
easy to check.

\noindent $(c)$  This comes directly from Proposition 
\ref{ring of endo}.

\noindent $(d)$ If $A$ is a division ring, $R(t-a)$ is a maximal left ideal of $R$ and
 Schur's lemma shows that $End_R(R/R(t-a))$ is a division ring.  The other statements are clear from our earlier results. 
\end{proof}

\begin{remark}
{\rm
In a division ring $K$, a $(\si,\de)$-conjugacy class
$\Delta^{\si,\de}(a)$ can be seen as a projective space associated to
$K$ considered as a right $C^{\si,\de}(a)$-vector space.  With
this point of view, for $f(t)\in R=K[t;\si,\de]$ without roots in
$\Delta^{\si,\de}(a)$, the projective map associated to the right $C^{\si,\de}
(a)$-linear map $f(T_a)$ is the map $\phi_f$ defined by
$\phi_f(a^x)=(a^x)^{f(a^x)}=a^{f(T_a(x))}$.  This map $\phi_f$ is useful
 for detecting pseudo-roots of a polynomial (i.e. elements $a\in K$
such that $t-a$ divides $gf\in R$ but $f(a)\ne 0$).  This point
of view sheds some lights on earlier results on $\phi$-transform
(Cf. \cite{LL3}).
}
\end{remark}


\begin{examples}
\label{examples usage of product formula}
{\rm
\begin{enumerate}
\item If $b-a\in A$ is invertible, it is easy to check that the 
polynomial $f(t):=(t-b^{b-a})(t-a)\in R=A[t;\si,\de]$ is a monic polynomial right
divisible by  $t-a$ and $t-b$.  $f(t)$ is thus the least left
common multiple (abbreviated LLCM in the sequel) of $t-a$ and $ t-b$ in
$R=A[t;\sigma,\delta]$.  Pursuing this theme further leads, in
particular, to noncommutative symmetric functions (Cf. \cite{DL}).
\item Similarly one easily checks that, if $f(a)$ is invertible then the
LLCM of $f(t)$ and $t-a$ in $R=A[t;\sigma,\delta]$ is given by
$(t-a^{f(a)})f(t)$.
\item It is now easy to construct polynomials that factor
completely in linear terms but have only one (right) root.  Let
$K$ be a division ring and $a\in K$ be an element algebraic
of degree two over the center $C$ of $K$.  We denote by $f_a(t)\in
C[t]$ the minimal polynomial of $a$. $f_a(t)$ is also the minimal
polynomial of the algebraic conjugacy class
$\Delta(a):=\{xax^{-1}\,|\,x\in K\setminus \{0\}\}$.   For $\gamma
\in \Delta(a)$, we note $\overline{\gamma}$ the unique element of
$K$ such that $f_a(t)=(t-\overline{\gamma})(t-\gamma)$. Let us
remark that if $\gamma\ne a$ then
$\overline{\gamma}=a^{a-\gamma}$.  
Using an induction on $m$, the reader can easily prove that if a polynomial
$g(t)$ is such that $g(t):=(t-a_m)(t-a_{m-1})\dots (t-a_1)$ where
$a_i\in \Delta(a)$ but $a_{i+1}\ne \overline{a_i}$, for
$i=1,\dots,m-1$ then $a_1$ is the unique root of $g(t)$.  
For a concrete example consider
$\mathbb H$, the division ring of quaternions over $\mathbb Q$. In
this case, for $a\in \mathbb H$, $\overline{a}$ is the usual
conjugate of $a$. Of course, one can generalize this example to a
$(\sigma,\delta)$-setting by considering an algebraic conjugacy
class of rank 2.
\item Let us describe all the irreducible polynomials
of $R:=\mathbb C[t;-]$.  First notice that the left (and right) Ore quotient ring 
$\mathbb C(t;-)$ of $R$ is a division ring of dimension $4$ over its center 
$\mathbb R(t^2)$.
This implies that any $f(t)\in \mathbb C[t;-]\setminus \mathbb R[t^2]$ satisfies an 
equation of the form:
$f(t)^2+a_1(t^2)f(t)+a_0(t^2)=0$ for some $a_1(t^2), a_0(t^2)\in \mathbb R(t^2)$ 
with $a_0(t^2)\ne 0$.  
This shows that for any polynomial $f(t) \in \mathbb C[t;-]\setminus 
\mathbb R[t]$ there exists $g(t)\in \mathbb C[t;-]$ such that 
$g(t)f(t)\in \mathbb R[t^2] \subset \mathbb R[t] \subset \mathbb C[t;-]$.  
In particular, the irreducible factors of $g(t)f(t)$ in $\mathbb 
C[t;-]$ are 
of degree $\le 2$.  
We can now conclude that the monic irreducible non linear polynomials 
of $\mathbb C[t;-]$ are the polynomials of the form $t^2+at+b$ with no (right) 
roots.   In other words the monic irreducible non linear polynomials of 
$\mathbb C[t;-]$ are of the form $t^2+at+b$ such that for any $c \in \mathbb C$, 
$c\overline{c}+ac +b \ne 0$. 

\end{enumerate}
}
\end{examples}

\vspace{2mm}

\noindent We now collect a few more observations. 

\begin{prop}
\label{pgcd and ppcm of PLT} 
Let $f,g\in R=A[t;\sigma,\delta]$ be polynomials such that $g$ is not a zero 
divisor and
$Rf+Rg=R$.   Suppose that there exists $m\in R$ with $Rm=Rf \cap Rg$.
Let $f',g'\in R$ be such that $m=f'g=g'f$.  
Let also $T$ be any pseudo-linear 
transformation.  We have:
\begin{enumerate}
\item[a)] $R/Rf'\cong R/Rf$.
\item[b)] $g(T)(\ker f(T))=\ker f'(T)$.
\item[c)] $\ker (m(T))=\ker f(T) \oplus \ker g(T)$.
\end{enumerate}
\end{prop}
\begin{proof}
$a)$ The morphism  $\varphi: R/Rf' \longrightarrow R/Rf$  of left $R$-modules 
defined by $\varphi (1 + Rf')= g+Rf$ is in fact an isomorphism.

\noindent $b)$ Since $f'g=g'f$, we have $(f'g)(T)(\ker f(T))=0$.  Hence $g(T)(\ker f(T))
\subseteq \ker f'(T)$.  Let $\varphi$ be the map defined in the proof of $a)$ above and let $h\in R$ 
be such that $\varphi^{-1}(1+Rf)=h +Rf'$.  Since 
$\varphi^{-1}$ is well defined, we have $fh\in Rf'$ and $h(T)(\ker f'(T)\subseteq 
\ker f(T)$.  We also have $gh-1\in Rf'$ and so $(gh)(T)\vert _{\ker f'(T)}=id.
\vert_{\ker f'(T)}$.  
This gives $\ker f'(T)=gh(T) (\ker f'(T))\subseteq g(T) (\ker f(T))\subseteq 
\ker f'(T)$.  This yields the desired conclusion.

\noindent $c)$ Obviously $\ker g(T) + \ker f(T) \subseteq \ker (m(T))$.
Now let $v\in \ker m(T)$.  Then $f'g(T)(v)=0=g'f(T)(v)$.  
This gives $g(T)(v)\in \ker f'(T)$ 
and so, using the equality $b)$ above, we have $g(T)(v)\in g(T)(\ker f(T))$.  
This shows that there exists $w\in \ker f(T)$ such that $g(T)(v)=g(T)(w)$.  We conclude
$v-w\in \ker g(T)$ and $v\in \ker g(T) + \ker f(T)$.  The fact that the sum is direct is clear from the 
equality $R=Rf + Rg$.
\end{proof}

As an application of the preceding proposition, we have a relation 
between the roots of two similar polynomials with coefficients in a division ring.  
For $f\in K[t;\si,\de]$, where $K$ is a division ring,  
we denote by $V(f)$ the set of right roots of $f$.  
For $x\notin V(f)$  we put $\phi_f(x):=x^{f(x)}:=\si(f(x))xf(x)^{-1}+\de(f(x))f(x)^{-1}$.  With these 
notations we have the following corollary of the previous proposition:

\begin{corollary}
Let $f,f'\in K[t;\si,\de]$ be such that $\varphi: R/Rf' \longrightarrow R/Rf$ is an 
isomorphism defined by $\varphi (1+Rf')=g+Rf$.  Then $V(f')=\phi_g(V(f))$.
\end{corollary}
\begin{proof}

Since $Rf+Rg=R$, $g(x)\ne 0 $ for any $x\in V(f)$ and we have: $f'(\phi_g(x))
g(x)=(f'g)(x)=(g'f)(x)=0$.   This shows that $\phi_g(V(f))\subseteq V(f')$.  
For the reverse inclusion let us remark that $y\in V(f')$ implies that $1\in 
\ker f'(T_y)$ the assertion $b)$ in the above proposition \ref{pgcd and ppcm of 
PLT} shows that there exists $z\in \ker f(T_y)$ such that $1= g(T_y)
(z)=g(y^z)z$.  An easy computation then gives that $y=\phi_g(y^z)$.  Since 
 $f(T_y)(z)=0$ implies $f(y^z)=0$, we conclude that $V(f')\subseteq \phi_g(V(f))$, as required.  
\end{proof}

\section{Applications}

Statement 1 of the following theorem is more
general and more precise than the classical
Gordon-Motzkin result (which is statement 1 of Theorem \ref{generalization of Gordon-Motzkin}  
with $(\sigma,\delta)=(id., 0)$ ).  This was already
mentioned in \cite{LLO} but we will state it in the language of
the maps $T_a$ introduced in Section $1$.   
For an element $a$ in a division ring $K$, we define 
$C^{\sigma,\delta}(a):=\{0\ne x\in K \,|\, a^x=a
\}\cup \{0\}$ (Cf. Section 1) and
$\Delta^{\si,\de}(a):=\{x\in K\setminus \{0\} \, | \, \si(x)a +\de(x)=ax \}$.  
$C^{\si,\de}(a)$ is a subdivision ring of $K$ and for any $f(t)\in R=K[t;\si,\de]$, 
$f(T_a)$ is a right $C^{\sigma,\delta}(a)$-linear map 
(Cf. Corollary \ref{the case of T_a}(d)).  
The set $\Delta(a)=\Delta^{\si,\de}(a)$ is the $(\si,\de)$-conjugacy class determined by $a$.
\begin{theorem}
\label{generalization of Gordon-Motzkin}

Let $f(t)\in R=K[t;\sigma,\delta]$ be a polynomial of degree $n$.
Then:
\begin{enumerate}
\item[1)] $f(t)$ has roots in at most $n$ $(\sigma,
\delta)$-conjugacy classes, say $\{\Delta
(a_1),\dots,\Delta(a_r)\}$, $r\le n$;
\item[2)] $\sum_{i=1}^rdim_{C(a_i)}\ker (f(T_{a_i}))\le n$, where 
$C(a_i):=C^{\si,\de}(a_i)$ for $1\le i \le r$.
\end{enumerate}
\end{theorem}
\begin{proof}
We refer the reader to \cite{LLO} and \cite{LO}.
\end{proof}

\begin{remark}
In \cite{LLO} it is shown that equality in formula $2)$ holds if and
only if the polynomial $f(t)$ is Wedderburn.
\end{remark}

We now offer an application of the previous Theorem
\ref{generalization of Gordon-Motzkin}.

 In coding theory some authors have used Ore extensions to define
noncommutative codes (Cf. \cite{BU}, \cite{BGU}).  In particular,
letting $\mathbb F_q$ be the finite field of characteristic $p$ with
$q=p^n$ elements, they considered the Ore extension of the form
$\mathbb F_q[t;\theta]$, where $\theta$ is the usual Frobenius
automorphism given by $\theta (x)=x^p$.  The following theorem
shows that the analogue of the usual minimal polynomial
$X^q-X \in \mathbb F_q[X]$ annihilating $\mathbb F_q$ is of much
lower degree in this noncommutative setting.

\begin{theorem}
\label{finite fields}

Let $p$ be a prime number and $\mathbb F_q$ be the finite field
with $q=p^n$ elements.  Denote by $\theta$ the Frobenius
automorphism.  Then:
\begin{enumerate}
\item[a)] There are $p$ distinct $\theta$-conjugacy classes in $\mathbb F_q$.
\item[b)] $C^{\theta}(0)=\mathbb F_q$ and, for $0\ne a\in \mathbb F_q$, we have
 $C^{\theta}(a)=\mathbb F_p$.
\item[c)] In $\mathbb F_q[t;\theta]$, the least left common multiple of all the elements
of the form $t-a$ for $ a\in\mathbb F_q$ is the polynomial
$G(t):=t^{(p-1)n +1}-t$.  In other words, $G(t)\in \mathbb
F_q[t;\theta]$ is of minimal degree such that $G(a)=0$ for all
$a\in \mathbb F_q$.
\item[d)] The polynomial $G(t)$ obtained in c) above is invariant,
i.e. $RG(t)=G(t)R$.
\end{enumerate}
\end{theorem}
\begin{proof}
a)   Let us denote by $g$ a generator of the cyclic group $\mathbb
F_q^{*}:=\mathbb F_q \setminus\{0\}$. The $\theta$-conjugacy class
determined by the zero element is reduced to $\{0\}$ i.e. $\Delta (0)=\{0\}$.
The $\theta$-conjugacy class determined by $1$ is a subgroup of $\mathbb
F_q^*$: $\Delta (1)=\{\theta (x)x^{-1}\,|\, 0\ne x\in \mathbb F_q
\}=\{x^{p-1}\,|\, 0\ne x\in \mathbb F_q\}$.  It is easy to check
that $\Delta (1)$ is cyclic generated by $g^{p-1}$ and has order
$\frac{p^n-1}{p-1}$.  Its index is $(\mathbb F_q^* : \Delta
(1))=p-1$.   Since two nonzero elements $a,b$ are
$\theta$-conjugate if and only if $ab^{-1}\in \Delta (1)$,  we
indeed get that the number of different nonzero $\theta$-conjugacy classes
is $p-1$.   This yields the result.

\noindent b) If $a\in \mathbb F_q$ is nonzero, then
$C^{\theta}(a)=\{x\in \mathbb F_q\,|\, \theta(x)a=ax\}$ i.e.
$C^{\theta}(a)=\mathbb F_p$.

\noindent c)  We have, for any $x\in \mathbb F_q,\,
(t^{(p-1)n+1}-t)(x)=\theta^{(p-1)n}(x)\dots \theta(x)x - x$. Since
$\theta^n=id.$ and $N_n(x):=\theta^{n-1}(x)\dots
\theta(x)x \in \mathbb
F_p$ , we get $(t^{(p-1)n+1}-t)(x)=
x(\theta^{n-1}(x)\dots\theta(x)x )^{p-1}-x=
xN_n(x)^{p-1}-x=0$.  This shows that indeed $G(t)$ 
annihilates all the elements
of $\mathbb F_q$ and hence $G(t)$ is a  left common multiple of
the linear polynomials $\{(t-a)\,|\,a\in \mathbb F_q\}$.  Let $h(t):=[t-a\,|\,a\in \mathbb F_q]_l$ denote their least left common multiple. It remains to show that
$\deg h(t)\ge n(p-1)+1$.  Let $0=a_0,a_1,\dots,a_{p-1}$ be
elements  representing the $\theta$-conjugacy classes (Cf. a)
above).  Denote by $C_0,C_1,\dots,C_{p-1}$ their respective
$\theta$-centralizer.  Corollary \ref{the case of T_a}(b) shows that 
$h(T_a)(x)=h(a^x)x=0$ for any nonzero element $x\in \mathbb F_q$ and any element 
$a\in\{a_0,\dots,a_{p-1}\}$.   
Hence $\ker h(T_{a_i})=\mathbb F_q$ for $0\le i \le p-1$. Using
the inequality $2)$ in Theorem \ref{generalization of
Gordon-Motzkin} and the statement $b)$ above, we get $\deg h(t)\ge
\sum_{i=0}^{p-1} dim _{C_i}\ker h(T_{a_i})=dim_{\mathbb F_q}\mathbb
F_q + \sum_{i=1}^{p-1}dim_{\mathbb F_p}\mathbb F_q =1+(p-1)n$, as
required.

\noindent d)  Since $\theta^n=id.$, we have immediately that
$G(t)x=\theta(x)G(t)$ and obviously $G(t)t=tG(t)$.
\end{proof}

\begin{remark}
The polynomial $G(t)=t^{n(p-1)+1}-t\in \mathbb F_{p^n}[t;\theta]$
defined in the previous theorem \ref{finite fields} can have roots
in an extension $\mathbb F_{p^l} \varsupsetneq\mathbb F_{p^n}$.  This
is indeed always the case if $l=n(p-1)$.  Let us denote by
$\Delta_l(1):=\{1^x\,|\, 0\ne x\in \mathbb F_{p^l}\}$ and
$\Delta_n(1):=\{1^x\,|\, 0\ne x\in \mathbb F_{p^n}\}$.   Since
$\theta^l=id.$ on $\mathbb F_{p^l}$, we have $G(t)a=\theta(a)G(t)$
for any $a\in \mathbb F_{p^l}$. This gives, for any $0\ne x\in
\mathbb F_{p^l}$
$G(1^x)x=(G(t)x)(1)=(\theta(x)G(t))(1)=\theta(x)G(1)=0$.  In other
words $G(t)$ annihilates the $\theta$-conjugacy class $\Delta_l
(1)\subseteq \mathbb F_{p^l}$.  It is easy to check that
$|\Delta_l(1)|=\frac{p^l-1}{p-1}>\frac{p^n-1}{p-1}=|\Delta_n(1)|$.
We conclude that $G(t)$ has roots in $\mathbb F_{p^l}\setminus
\mathbb F_{p^n}$.  This contrasts with the classical case where
$[x-a\,|\,a\in \mathbb F_{p^n}]_l=x^{p^n}-x\in \mathbb
F_{p^n}[x]$ has all its roots in $\mathbb F_{p^n}$.
\end{remark}

For a prime $p$ and an integer $i\ge 1$, we define $[i]:=\frac{p^i-1}
{p-1}=p^{i-1}+p^{i-2}+\dots +1$ and put $[0]=0$.  We fix an integer $n\ge 1$ and continue to denote $q=p^n$.  Let us introduce the following subset of $\mathbb F_q[x]$:

$$
\mathbb F_q[x^{[]}]:=\{ \sum_{i\ge 0} \alpha_i x^{[i]} \in \mathbb F_q[x] \}
$$

A polynomial belonging to this set will be called a $[p]$-polynomial.
We extend $\theta$ to the ring $\mathbb F_q[x]$ and put $\theta (x)=x^p$ i.e. $\theta (g)=g^p$ 
for all $g\in \mathbb F_q[x]$.  We thus have $R:=\mathbb F_q[t;\theta]\subset S:=\mathbb F_q[x][t;\theta]$.
Considering $f\in R:=\mathbb F_q[t;\theta]$ as an element of 
$\mathbb F_q[x][t;\theta]$ we can evaluate $f$ at $x$.  
We denote the resulting polynomial by 
$f^{[]}[x]\in \mathbb F_q[x]$ i.e. $f(t)(x)=f^{[]}(x)$.  

The last statement of the following theorem will show that 
the question of the irreducibility of a polynomial $f(t)\in R:=\mathbb F_q[t;\theta]$ can 
be translated in terms of factorization in $\mathbb F_q[x]$.  This makes Berlekamp 
algorithm available to test irreducibility of polynomials in $R=\mathbb F_q[t;\theta]$.  
This will also provide an algorithm for factoring polynomials in $\mathbb F_q[t;\theta]$,
as explained in the paragraph following the proof of the next theorem.

\begin{theorem}
\label{factorization without theta}
Let $f(t)=\sum_{i=0}^na_it^i$ be a polynomial in $R:=\mathbb F_q[t;\theta]\subset  
S:=\mathbb F_q[x][t;\theta]$.   With the above notations we have:
\begin{enumerate}
\item[1)] For any $h=h(x)\in \mathbb F_q[x],\; f(h)=\sum_{i=0}^n a_ih^{[i]}$. 
\item[2)] $\{f^{[]}\vert f \in R=\mathbb F_q[t;\theta] \}=\mathbb F_q[x^{[]}]$.
\item[3)] For $i\ge 0$  and  $h(x)\in \mathbb F_q[x]$ we have $T_x^i(h)=h^{p^i}x^{[i]}$.
\item[4)] For $g(t)\in S=\mathbb F_q[x][t;\theta]$ and $h(x)\in \mathbb F_q[x]$ we have 
$g(T_x)(h(x))\in \mathbb F_q[x]h(x)$.
\item[5)] For any $h(t)\in R=\mathbb F_q[t;\theta]$, $f(t) \in Rh(t)$ if and only 
if $f^{[]}(x) \in \mathbb F_q[x]h^{[]}(x)$.
\end{enumerate}
\end{theorem}
\begin{proof}
\noindent 1)
 We compute: $f(t)(h)=(\sum_{i=0}^na_it^i)(h)=
\sum_{i=0}^na_i\theta^{i-1}(h)\cdots \theta (h)h=\sum_{i=0}^na_ih^{[i]}$.  

\noindent 2) This is clear from the statement 1) above for $h=x$.

\noindent 3) This is easily proved by induction (notice that 
$T_x^0(h)=h=h^{p^0}x^{[0]}$).

\noindent 4) Let us put $g(t)=\sum_{i=0}^ng_i(x)t^i$.  Statement 3) above then gives: 
$g(T_x)(h(x))=
(\sum_{i=0}^ng_i(x)T_x^i)(h(x))=\sum_{i=0}^ng_i(x)h^{p^i}x^{[i]}\in \mathbb F_q[x]h$.

\noindent 5) Let us write $f(t)=g(t)h(t)$ in $R$.   Corollary \ref{the case of T_a} (a)   
and statement 4) above give $f^{[]}(x)=f(t)(x)=
(g(t)h(t))(x)=g(T_x)(h(t)(x))=g(T_x)(h^{[]}(x))\in \mathbb F_q[x]h^{[]}(x)$.

\noindent Conversely, suppose there exists $g(x) \in \mathbb F_q[x]$ such that 
$f^{[]}(x)=g(x)h^{[]}(x)$.  Let $f(t), h(t) \in \mathbb F_q[t;\theta]$ be such that 
$f(t)(x)=f^{[]}(x)$ and $h(t)(x)=h^{[]}(x)$.  Using the euclidean division algorithm
in $\mathbb F_q[t;\theta]$ we can write $f(t)=q(t)h(t) + r(t)$ 
with $\deg r(t) < \deg h(t)$.  Evaluating both sides of this equation at $x$ 
we get, thanks to the generalized product formula, 
$f^{[]}(x)=f(t)(x)=q(T_x)(h(t)(x))+r(t)(x)=q(T_x)(h^{[]}(x))+r^{[]}(x)$ and 
$\deg r^{[]}(x)=[\deg r(t)] < [\deg h(t)]= \deg h^{[]}(x)$.  
Statement 4) above and the hypothesis then give that $r^{[]}(x)=0$.  
Let us write $r(t)=\sum_{i=0}^lr_it^i\in \mathbb F_q[t;\theta]$.  With these notations we 
must have $\sum_{i=0}^lr_ix^{[i]}=0$.  This yields that for all $i\ge 0, \;r_i=0$ and 
hence $r(t)=0$, as required.   
\end{proof}

Let us mention the following obvious but important corollary:

\begin{corollary}
\label{criterion of irreducibility for Ore polynomials over finite fields}
A polynomial $f(t)\in \mathbb F_q[t;\theta]$ is irreducible if and only if its 
attached $[p]$-polynomial $f^{[]}\in \mathbb F_q[x^{[]}]\subset 
\mathbb F_q[x]$ has no non trivial factor belonging to 
$\mathbb F_q[x^{[]}]$.
\end{corollary}
  
Of course, the condition stated in the above corollary \ref{criterion of 
irreducibility for Ore polynomials over finite fields} can be checked using, for 
instance, the Berlekamp algorithm for factoring polynomials over finite fields.  
This leads easily to an algorithm 
for factoring $f(t) \in \mathbb F_q[t;\theta]$.  
Indeed given $f(t)\in \mathbb F_q[t;\theta]$ we first find a polynomial 
$h^{[]}\in \mathbb F_q[x^{[]}]$ such that $h^{[]}$ 
divides $f^{[]}$ (if possible) and we write 
$f^{[]}=g(x)h^{[]}$ for some $g(x)$ in $\mathbb F_q[x]$.
This gives $f(t)=g'(t)h(t) \in \mathbb F_q[t;\theta]$.  We then apply the same procedure 
to $g'(t)$ and find a right factor of $g'(t)$ in $\mathbb F_q[t;\theta]$ 
by first finding (if possible) a $[p]$-factor of $g'^{[]}$...
Let us give some concrete examples.

\begin{examples}
In the next three examples we will consider the field of four elements 
$\mathbb F_4=\{0,1,a,1+a\}$ where $a^2+a+1=0$.  $\theta (a)=a^2=a+1;\; 
\theta(a+1)=(a+1)^2=a$.
\begin{enumerate}
\item[a)] Consider the polynomial $t^3+a \in \mathbb F_4[t;\theta]$.
Its associated $[2]$-polynomial is given by $x^7+a\in \mathbb F_4[x]$.  
Since $a$ is a root of $x^7+a$ it is also a root of $t^3+a$.  This gives
$t^3+a=(t^2+at+1)(t+a)$ in $\mathbb F_4[t;\theta]$.  Now, the $[2]$-polynomial 
associated to the left factor $t^2+at+1$ is $x^3+ax+1\in \mathbb F_4[x]$.  Since this last 
polynomial is actually irreducible we conclude that $t^2+at+1$ is also irreducible in 
$\mathbb F_4[t;\theta]$.  Hence the factorization of $t^3+a$ given above is in fact
a decomposition into irreducible polynomials.

\item[b)] Let us now consider $f(t)=t^4+(a+1)t^3+a^2t^2+(1+a)t+1\in \mathbb F_4[t;
\theta]$.  Its attached $[p]$-polynomial is 
$x^{15}+(a+1)x^7+(a+1)x^3+(1+a)x+1 \in \mathbb F_4[x]$.  
We can factor it as follows: 
$$
(x^{12}+ax^{10}+x^9+(a+1)x^8+(a+1)x^5+(a+1)x^4+x^3+ax^2+x+1)(x^3+ax+1)
$$ 
This last factor is a $[p]$-polynomial which corresponds to 
$t^2+at+1\in \mathbb F_4[t;\theta]$.  Moreover since $x^3+ax+1$ is  
irreducible in $\mathbb F_4[x]$, $t^2+at+1$ is also irreducible in 
$\mathbb F_4[t;\theta]$.  We then easily conclude that $f(t)=(t^2+t+1)(t^2+at+1)$
is a decomposition of $f(t)$ into irreducible factors in  $\mathbb F_4[t;\theta]$.

\item[c)]  Let us consider the polynomial $f(t)=t^5+at^4+(1+a)t^3+at^2+t+1$.
Its attached $[p]$-polynomial is $x^{31}+ax^{15}+(1+a)x^7+ax^3+x+1$.  It is easy to remark 
that $a$ is a root and we get $f(t)=q_1(t)(t+a)$ in $\mathbb F_4[t;\theta]$ where 
$q_1(t)=t^4+(a+1)(t^2+t+1)$.   The $[p]$-polynomial attached to $q_1(t)$ is $x^{15}
+(a+1)(x^3+x+1)$.  Again we get that $a$ is a root and we obtain that 
$q_1(t)=(q_2(t))(t+a)$ in $\mathbb F_4[t;\theta]$ where $q_2(t)=t^3+(a+1)t^2+at+a$.
The $[p]$-polynomial attached to $q_2(t)$ is $x^7+(a+1)x^3+ax+a$.  Once again $a$ is a 
root and we have $q_2(t)=(t^2+t+1)(t+a)$.  Since $t^2+t+1$ is easily seen to be 
irreducible in $\mathbb F_4[t;\theta]$, we have the following factorization of our 
original polynomial: $f(t)=(t^2+t+1)(t+a)^3$.  We can also factorize $f(t)$ as follows: 
$f(t)=(t+a+1)(t+1)(t+a)(t^2+(a+1)t+1)$.
\end{enumerate}
\end{examples}

\begin{remark}
It is a natural question to try to find a good notion of a splitting field attached 
to a 
polynomial of an Ore extension.  The above results justify that, in the case of
a skew polynomial ring $\mathbb F_q[t;\theta]$ where $q=p^n$ and $\theta$
is the Frobenius automorphism, we define the splitting field of a polynomial 
$f(t)\in \mathbb F_q[t;\theta]$ to be the splitting of the polynomial $f^{[]}(x)$
over $\mathbb F_q$.  
\end{remark}

Our next application of Theorem \ref{generalization of
Gordon-Motzkin}, is an easy proof of Hilbert 90 theorem (Cf.
\cite{LL3} for more advanced results on Hilbert 90 theorem in a
($\sigma,\delta$) setting).

\begin{prop}
\label{hilbert 90}
\begin{enumerate}
\item[a)] Let $K$ be a division ring, $\sigma$ an automorphism of $K$
 of finite order $n$ such that no power of $\si$ of order strictly smaller
 than $n$ is inner.  Then $\Delta^{\sigma}(1)$ is algebraic and
$t^n-1\in K[t;\sigma]$ is its minimal polynomial (i.e.
$V(t^n-1)=\Delta^{\sigma} (1)$).
\item[b)] Let $K$ be a division ring of characteristic $p>0$ and $\delta$ a
nilpotent derivation of $K$ of order $p^n$ satisfying no identity
of smaller degree than $p^n$.  Then $\Delta^{\delta}(0)$ is
algebraic and $t^{p^n}$ is its minimal polynomial
$(V(t^{p^n})=\Delta^{\delta}(0) )$.
\end{enumerate}
\end{prop}
\begin{proof}
a)  Since $T_1^n=\si^n=id.$, we have $\ker(T_1^n-id.)=K$. It is
easy to check that $(t^n-1)(\si(x)x^{-1})=0$ for any $x\in
K\setminus \{0\}$.  We thus have $\Delta^{\si}(1)\subseteq V(t^n-1)$.
 Standard Galois theory of division rings implies that
$[K:Fix(\si)]_r=n$.  Moreover $C^{\si}(1)=Fix(\si)$, part two of
Theorem \ref{generalization of Gordon-Motzkin} than quickly yields
the result.

\noindent b)  This is similar to the above proof noting that
$K=\ker(\de^{p^n})=\ker(T_0^{p^n})$, $C^{\de}(0)=\ker(\de)$ and
$[K:\ker(\de)]_r=p^n$.
\end{proof}

\begin{remark}
We do get back the standard Hilbert 90 theorem remarking in
particular that $\Delta^{\si}(1)=\{\si(x)x^{-1}\,|\,x\in
K\setminus\{0\} \}$.
\end{remark}

As another application, let us now give a quick proof of a
generalized version of the Frobenius formula in characteristic
$p>0$.   The proof of this formula is usually given for a field
through long computations involving additive commutators (Cf.
Jacobson \cite{Ja1}, p. 190). Using polynomial maps we get a shorter proof.

\begin{prop}
Let $K$ be a ring of characteristic $p>0$, $\delta$ be a (usual)
derivation of $K$ and $a$ any element in $K$.  In $R=K[t;id.,\delta]$ we
have
$$
(t-a)^p=t^p-T_a^p(1).
$$
\end{prop}
\begin{proof}
Define a derivation $d$ on $R$ by $d|_K=0$ and $d(t)=1$.  It is
easy to check that this gives rise to a well defined derivation on
$R$.  Notice that $d(t-a)=1$ commutes with $t-a$ hence
$d((t-a)^p)=0$.  Let us write $(t-a)^p=\sum_{i=0}^pc_it^i$.
Applying $d$ on both sides we quickly get that $c_i=0$ for all
$i=1,\dots,p-1$.  We thus have $(t-a)^p=t^p-c_0$.  Since $a$ is a
right root we indeed have that $c_0=t^p(a)=T_a^p(1)$.
\end{proof}

Let us now analyze the maps arising in a division process. For typographical reasons it 
is convenient to write, for $a\in A$ and $i\ge 0$, $N_i(a):=T_a^i(1)$.  
Properties of these maps can be found in 
previous works (e.g. \cite{LL1}, \cite{LL2}).  Here we
will look at the quotients and get some formulas generalizing
elementary ones.  It doesn't seem that these maps have been
introduced earlier in this setting.

\begin{prop}
\label{the quotients} Let $A,\si,\de$ be a ring, an endomorphism
and a $\si$-derivation of $A$, respectively.  For $a\in A$ and $i\ge 0$, let us
write $t^i=q_{i,a}(t)(t-a)+N_i(a)$ in $R=A[t;\si,\de]$.  We have:
\begin{enumerate}
\item[1)] If $f(t)=\sum_{i=0}^na_it^i\in R$, then
$f(t)=\sum_{i=0}^na_iq_{i,a}(t)(t-a)+\sum_{i=0}^na_iN_i(a)$.
\item[2)] $q_{0,a}=0,\; q_{1,a}=1$ and, for $i\ge 1$,
$q_{i+1,a}(t)=tq_{i,a}(t)+\si(N_i(a))$.
\item[3)] $N_i(b)-N_i(a)=q_{i,a}(T_b)(b-a)=q_{i,b}(T_a)(a-b)$.
\end{enumerate}
\end{prop}
\begin{proof}
The elementary proofs are left to the reader.
\end{proof}

\begin{remark}
Even the case when $\si=id.$ and $\de=0$ is somewhat interesting.
In this case the polynomials $q_{i,a}$ can be expressed
easily: $q_{i,a}(t)=t^{i-1}+at^{i-1}+\cdots + a^{i-1}$.  Of
course, we also get some familiar formulas. For instance
the last equation in \ref{the quotients} above gives the classical
equality in a noncommutative ring $A$:
$b^i-a^i=(b-a)b^{i-1}+a(b-a)b^{i-2}+\cdots + a^{i-1}(b-a)$.
\end{remark}

%

We now present the last application which is related to the case
when the base ring is left duo.
\begin{prop}
\label{ LLCM of 2 linear poly. and left duo}

Let $A,\si,\de$ be respectively, a ring, an endomorphism of $A$ and a
$\si$-derivation of $A$.   The following are equivalent:
\begin{enumerate}
\item[(i)] For $a,b\in A$, there exist $c,d\in A$ such that
$(t-c)(t-a)=(t-d)(t-b)$ in $R=A[t;\si,\de]$;
\item[(ii)] For any $a,b\in A$, there exists $c\in A$ such that $T_b(a)=ca=L_c(a)$;
\item[(iii)] For any $a,b\in A$, there exists $c\in A$ such that
$\si(a)b+\de(a)=ca$.
\end{enumerate}
In particular, when $\si=id.$ and $\de=0$, the above conditions
are also equivalent to the ring $A$ being left duo.
\end{prop}
\begin{proof}
\noindent $(i)\Rightarrow (ii)$.  Clearly $(i)$ implies that $b$
is a (right) root of $(t-c)(t-a)$.  Hence for every $a,b\in A$
there exists $c\in A$ such that $(T_b-c)(b-a)=0$.  Since $a,b$ are
any elements of $A$ this implies $(ii)$.

\noindent $(ii)\Rightarrow (iii)$.  This comes from
the definition of $T_b$.

\noindent $(iii)\Rightarrow (i)$.  Let $a,b\in A$.  Writing the
condition $(iii)$ for the elements $b-a$ and $b$ we find an
element $c\in A$ such that $\si(b-a)b+\de(b-a)=c(b-a)$.  We then
check that  $((t-c)(t-a))(b)=0$.   This shows that $(t-c)(t-a)$ is
right divisible by $t-b$ and this proves statement $(i)$.

\noindent The additional statement is clear from $(iii)$ indeed in
this case $(iii)$ means that for any $a,b\in A$, $ab\in Aa$.  Or
in other words, that any left principal ideal $Aa$ is in fact a two
sided ideal.
\end{proof}

The last statement of the previous proposition \ref{ LLCM of 2
linear poly. and left duo} justifies the following definition:

\begin{definition}
A ring $A$ is left $(\si,\de)$-duo if for any $a,b\in A$, there
exists $c\in A$ such that $T_b(a)=ca$.
\end{definition}

Proposition \ref{ LLCM of 2 linear poly. and left duo} was already given in the last 
section of \cite{DL}.
Here we stress the use of $T_a$.   In fact the pseudo-linear map
$T_a$ enables us to show that in an Ore extension built on a left
$(\si,\de)$-duo ring, the least left common multiple exists for
any two monic polynomials as long as one of them can be factorized
linearly.  We state this more precisely in the following theorem.
This theorem was also proved by M. Christofeul with a different,
more computational, proof \cite{C}.

\begin{theorem}
Let $a_1,\dots,a_n$ be elements in a left $(\si,\de)$-duo ring
$A$.  Then for any monic polynomial $g(t)\in R=A[t;\si,\de]$ there
exists a monic least left common multiple of $g(t)$ and of
$(t-a_n)\cdots(t-a_1)$ of degree $\le n+\deg (g)$.
\end{theorem}
\begin{proof}
We proceed by induction on $n$.  If $n=1$ the fact that $A$ is
$(\si,\de)$-left duo implies that there exists $c\in A$ such that
$T_{a_1}(g(a_1))=cg(a_1)$ and this shows that the polynomial
$(t-c)g(t)$ is divisible on the right by $t-a_1$, as desired.

\noindent Assume $n>1$.  By the above paragraph, there exist a
monic polynomial $g_1(t)\in R$ and an element $c\in A$ such that
$g_1(t)(t-a_1)=(t-c)g(t)$.  On the other hand,  the induction
hypothesis shows that there exist monic polynomials $h(t), p(t)\in
R$ such that $h(t)(t-a_n)\cdots (t-a_2)=p(t)g_1(t)$ where $\deg(h)
+ n - 1\le \deg(g_1)+n-1=\deg(g)+n-1$. This implies that
$h(t)(t-a_n)\cdots
(t-a_2)(t-a_1)=p(t)g_1(t)(t-a_1)=p(t)(t-c)g(t)$.  This shows that
$g(t)$ and $(t-a_1)(t-a_2)\cdots (t-a_n)$ have a monic common
multiple of degree $\le \deg(g)+n$, as desired.
\end{proof}

\end{document}